%% file: main.tex
\documentclass[10pt,a4paper]{article}

 \usepackage{float}
 \usepackage{a4wide}
 \usepackage{amssymb}
 \usepackage{amsthm}
 \usepackage{amsmath}
 \usepackage{cleveref}
 \usepackage{mathrsfs}
 \usepackage{wrapfig}
  \usepackage{graphicx}

\usepackage{authblk}

 \usepackage{tikz}
\usetikzlibrary{shapes,arrows,calc, fit, decorations}
\usetikzlibrary{decorations.pathreplacing}
\usetikzlibrary{fit}

\newtheorem{theorem}{Theorem}

\newtheorem{lemma}[theorem]{Lemma}

\newtheorem{conjecture}[theorem]{Conjecture}

\title{On the chromatic numbers of signed triangular and hexagonal grids}
\author{Fabien Jacques}
\affil{LIRMM, University of Montpellier, CNRS, Montpellier, France}
\date{}

\begin{document}
\maketitle

\begin{abstract}
A signed graph is a simple graph with two types of edges. Switching a vertex $v$ of a signed graph corresponds to changing the type of each edge incident to $v$.

A homomorphism from a signed graph $G$ to another signed graph $H$ is a mapping $\varphi: V(G) \rightarrow V(H)$ such that, after switching any number of the vertices of $G$, $\varphi$ maps every edge of $G$ to an edge of the same type in $H$. The chromatic number $\chi_s(G)$ of a signed graph $G$ is the order of a smallest signed graph $H$ such that there is a homomorphism from $G$ to $H$.

We show that the chromatic number of signed triangular grids is at most 10 and the chromatic number of signed hexagonal grids is at most 4.
\end{abstract}

 \input{introduction}

 \input{proof_hex}
 \input{proof_tri}

 \bibliographystyle{unsrt}
 \bibliography{main}

\end{document}

%% file: introduction.tex
\section{Homomorphisms of signed and 2-edge-colored graphs}
\label{sec:intro}

A \textit{2-edge-colored graph} or a \textit{signed graph} $G = (V, E, s)$ is a simple graph $(V, E)$ with two kinds of edges: positive and negative edges. We do not allow parallel edges nor loops. The signature $s: E(G) \rightarrow \{-1, +1\}$ assigns to each edge its sign. For the concepts discussed in this article, 2-edge-colored graphs and signed graphs only differ on the notion of homomorphisms.

Given two $2$-edge-colored graphs $G$ and $H$, the mapping $\varphi : V(G)\rightarrow V(H)$ is a \textit{homomorphism} if $\varphi$ maps every edge of $G$ to an edge of $H$ with the same sign. This can be seen as coloring the graph $G$ by using the vertices of $H$ as colors. The target graph $H$ gives us the rules that this coloring must follow. If vertices $1$ and $2$ of $H$ are adjacent with a positive (resp. negative) edge, then every pair of adjacent vertices in $G$ colored with $1$ and $2$ must be adjacent with a positive (resp. negative) edge. 

\textit{Switching} a vertex $v$ of a $2$-edge-colored or signed graph corresponds to reversing the signs of all the edges that are incident to $v$. 

Given two signed graphs $G$ and $H$, the mapping $\varphi : V(G)\rightarrow V(H)$ is a \textit{homomorphism} if there is a homomorphism from $G$ to $H$ after switching some subset of the vertices of $G$ and/or switching some subset of the vertices of $H$. However, switching in $H$ is unnecessary (as explained in Section 3.3 of \cite{HomSG}).

The \textit{chromatic number} $\chi_s(G)$ of a signed graph $G$ is the order of a smallest signed graph $H$ such that $G$ admits a homomorphism to $H$. 
The chromatic number $\chi_s(\mathcal{C})$ of a class of signed graphs $\mathcal{C}$ is the maximum of the chromatic numbers of the graphs in the class. \newline

Homomorphisms of signed graphs were introduced by Naserasr, Rollová and Sopena \cite{HomSG}. This type of homomorphism allows us to generalize several classical problems such has Hadwiger's conjecture \cite{HomSG, HomSG2} and have therefore been studied by a great number of authors. Here are several known results on the chromatic number of some classes of signed graphs that are related to the classes we study in this article.

\begin{theorem}
\label{thm:art}
\begin{enumerate} 
\item\label{thm:art-planar} The chromatic number of signed planar graphs is at most 40~\cite{OPS}.
\item\label{thm:art-planar6} The chromatic number of signed planar graphs with girth at least 6 is at most 6~\cite{Hom2ec}.
\item\label{thm:art-degree3} The chromatic number of signed graphs with maximum degree 3 is at most 7~\cite{BPS}.
\item\label{thm:art-mad3} The chromatic number of signed graphs with maximum average degree less than 3 is at most 6~\cite{BPS}.
\item\label{thm:art-2Dgrid} The chromatic number of signed square grids is at most 6~\cite{Grid}.
\end{enumerate}
\end{theorem}

In \Cref{sec:results} we present our results on the chromatic number of hexagonal and triangular grids and in \Cref{sec:targets} we introduce the target graphs that we use in \Cref{sec:tri,sec:hex} to prove these results.

\section{Results}
\label{sec:results}
A square (resp. triangular, hexagonal) grid is a finite induced subgraph of the graph associated with the tiling of the plane with squares (resp. equilateral triangles, equilateral hexagons). See \Cref{fig:tri_grid,fig:hex_grid}.
Since signed hexagonal grids have maximum degree~3 and therefore maximum average degree at most~3, we already know that their chromatic number is at most $6$ 
by Theorem~\ref{thm:art}(\ref{thm:art-degree3}) or by Theorem~\ref{thm:art}(\ref{thm:art-mad3}). Moreover, signed triangular grids are planar and have therefore chromatic number at most $40$ by Theorem~\ref{thm:art}(\ref{thm:art-planar}). We improve these bounds as follows.

\begin{theorem}
\label{thm:hex}
The chromatic number of signed hexagonal grids is 4.
\end{theorem}

\begin{theorem}
\label{thm:tri}
The chromatic number of signed triangular grids is at most 10.
\end{theorem}

In order to prove these theorems, we will show that every signed hexagonal grid admits a homomorphism to a target graph of order 4 we call $T_4$ (see Figure~\ref{fig:SP_5}) and that every signed triangular grid admits a homomorphism to a target graph of order 10 called $SP_{9}^+$. Constructions of $T_4$ and $SP_9^+$ are explained in Section~\ref{sec:targets}. Note that it is conjectured that every signed planar graph admits a homomorphism to $SP_{9}^+$~\cite{conj}. Theorem~\ref{thm:tri} brings further evidence toward this conjecture.

\section{Target Graphs}
\label{sec:targets}
A $2$-edge-colored graph $(V, E, s)$ is said to be \textit{antiautomorphic} if it is isomorphic to $(V, E, -s)$.

A $2$-edge-colored graph $G = (V, E, s)$ is said to be \textit{$K_n$-transitive} if for every pair of cliques $\{u_1, u_2, \ldots , u_n\}$ and $\{v_1, v_2, \ldots , v_n\}$ in $G$ such that $s(u_i u_j) = s(v_i v_j)$ for all $i \neq j$, there exists an automorphism that maps $u_i$ to $v_i$ for all $i$. 
For $n = 1$ or $2$, we say that the graph is \textit{vertex-transitive} or \textit{edge-transitive}, respectively. 

A 2-edge-colored graph $G$ has \emph{Property $P_{k, n}$} if for every sequence of $k$ distinct vertices $(v_1, v_2, \dots, v_k)$ that induces a clique in $G$ and for every sign vector $(\alpha_1, \alpha_2, ...,$ $\alpha_k) \in \{-1, +1\}^k$ there exist at least $n$ distinct vertices $\{u_1, u_2, ..., u_n\}$ such that $s(v_i u_j) = \alpha_i$ for $1 \leq i \leq k$ and $1 \leq j \leq n$.

\medmuskip=0mu

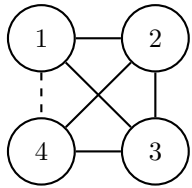
\begin{figure}[H]
	\centering
	\scalebox{1}
	{
		\begin{tikzpicture}[thick]
			\def \radius {3cm}
			\def \margin {8} 
			\tikzstyle{vertex}=[circle,minimum width=2.5em]
			
    		\node[draw, vertex] (4) at (0, 0) {$4$};
    		\node[draw, vertex] (1) at (0, 1.5) {$1$};
    		\node[draw, vertex] (2) at (1.5, 1.5) {$2$};
    		\node[draw, vertex] (3) at (1.5, 0) {$3$};
    		
    		\draw (1) -- (2);
    		\draw (1) -- (3);
    		\draw[dashed] (1) -- (4);
    		\draw (2) -- (3);
    		\draw (2) -- (4);
    		\draw (3) -- (4);
    		
			\end{tikzpicture}
	}
	\caption{The graph $T_4$.}
	\label{fig:SP_5}
\end{figure}

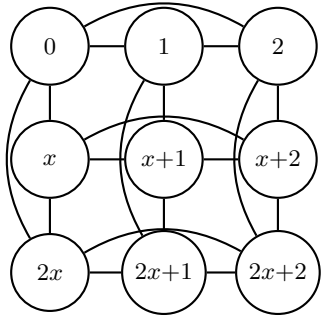
\begin{figure}[H]
	\centering
	\scalebox{1}
	{
		\begin{tikzpicture}[thick]
		\def \radius {1cm}
		\def \margin {8} 
		\tikzstyle{vertex}=[circle,minimum width=2.9em]
		
		\node[draw, vertex] (0) at (0, 3) {\small$0$};
		\node[draw, vertex] (1) at (1.5, 3) {\small$1$};
		\node[draw, vertex] (2) at (3, 3) {\small$2$};
		\node[draw, vertex] (3) at (0, 1.5) {\small$x$};
		\node[draw, vertex] (4) at (1.5, 1.5) {\small$x+1$};
		\node[draw, vertex] (5) at (3, 1.5) {\small$x+2$};
		\node[draw, vertex] (6) at (0, 0) {\small$2x$};
		\node[draw, vertex] (7) at (1.5, 0) {\small$2x+1$};
		\node[draw, vertex] (8) at (3, 0) {\small$2x+2$};

		\draw (0) -- (1);
		\draw (1) -- (2);
		\draw (2) edge[bend right]  (0);
		
		\draw (3) -- (4);
		\draw (4) -- (5);
		\draw (5) edge[bend right]  (3);
		
		\draw (6) -- (7);
		\draw (7) -- (8);
		\draw (8) edge[bend right]  (6);
		
		\draw (0) -- (3);
		\draw (3) -- (6);
		\draw (6) edge[bend left]  (0);
		
		\draw (1) -- (4);
		\draw (4) -- (7);
		\draw (7) edge[bend left]  (1);
		
		\draw (2) -- (5);
		\draw (5) -- (8);
		\draw (8) edge[bend left]  (2);
		
		\end{tikzpicture}
	}
	\caption{The graph $SP_9$, non-edges are negative edges.}
	\label{fig:SP_9}
\end{figure}

\medmuskip=4mu

The \textit{2-edge-colored Paley graph} $SP_9$ has vertex set \\ $V(SP_9) = \mathbb{F}_9$, the field of order 9. Two vertices $u$ and $v \in V(SP_9)$, $u \neq v$, are connected with a positive edge if $u - v$ is a square in $\mathbb{F}_9$ and with a negative edge otherwise.
Notice that this definition is consistent because $9 \equiv 1 \mod 4$ so $-1$ is always a square in $\mathbb{F}_9$ and if $u-v$ is a square then $v-u$ is also a square.

Given a 2-edge-colored graph $G$ with signature $s_G$, we create the \textit{antitwinned graph} of $G$ denoted by $\rho(G)$ as follows:

Let $G^{+1}$, $G^{-1}$ be two copies of $G$. The vertex corresponding to $v \in V(G)$ in $G^{i}$ is denoted by~$v^i$.

\begin{itemize}
\item $V(\rho(G)) = V(G^{+1}) \cup V(G^{-1})$
\item $E(\rho(G)) = \{ u^i v^j : uv \in E(G), \ i, j \in \{-1, +1\} \}$
\item $s_{\rho(G)}(u^i v^j) = i \times j \times s_G(u, v)$
\end{itemize}

By construction, for every vertex $v$ of $G$, $v^{-1}$ and $v^{+1}$ are \textit{antitwins}, the positive neighbors of $v^{-1}$ are the negative neighbors of $v^{+1}$ and vice versa. A 2-edge-colored graph is \textit{antitwinned} if every vertex has a unique antitwin.

When coloring a 2-edge-colored graph with an antitwinned graph, we say that two vertices have the same \textit{identity} if they are mapped to the same vertex or vertices that are antitwinned. In an antitwinned 2-edge-colored graph we denote the antitwin of $v$ with $\overline{v}$.

\begin{lemma}[\cite{HomEG}]
\label{lem:BG}
Let $G$ and $H$ be 2-edge-colored graphs. The two following propositions are equivalent:
\begin{itemize}
\item The 2-edge-colored graph $G$ admits a homomorphism to $\rho(H)$.
\item The signed graph $G'$ defined by the 2-edge-colored graph $G$ admits a homomorphism to the signed graph $H$.
\end{itemize}
\end{lemma}

Given a signed graph $H$, we define $H^+$ to be $H$ with an added universal vertex $\infty$ that is positively connected to all the other vertices (See \Cref{fig:Hplus}). We will use this construction to create the target graph $SP_9^+$ used in Section~\ref{sec:tri}.
\begin{figure}
\centering
\includegraphics[scale=0.2]{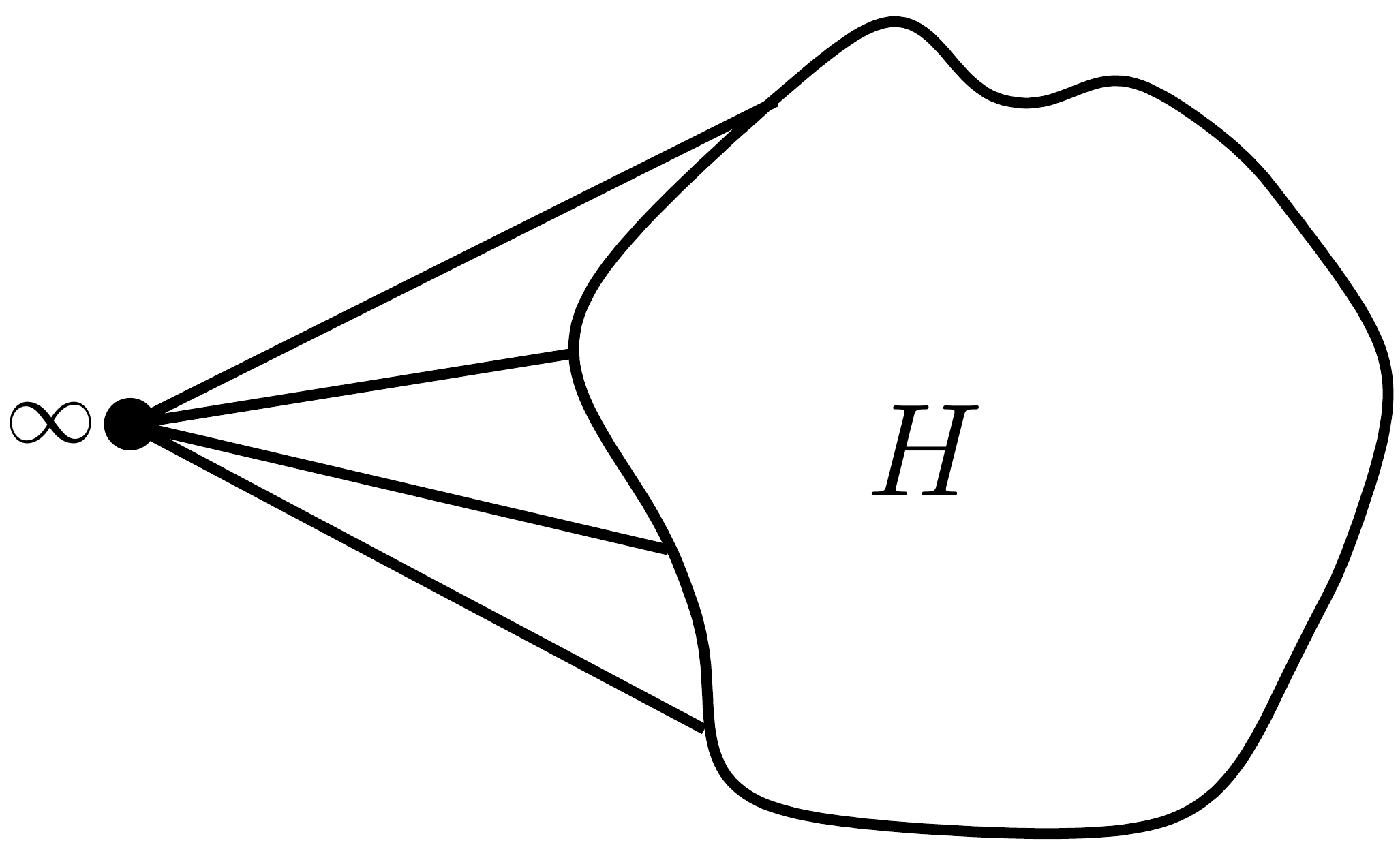}
\caption{The construction of $H^+$.}
\label{fig:Hplus}
\end{figure}
\newline

We will use the following properties of our target graphs to prove our theorems.

\begin{lemma}[\cite{OPS16}]
\label{lem:PTRSP}
The graph $\rho(SP_9^+)$ is vertex-transitive, edge-transitive, antiautomorphic and has Properties $P_{1, 9}$, $P_{2, 4}$ and $P_{3, 1}$.
\end{lemma}

\begin{lemma}
\label{lem:SP5}
The graph $\rho(T_4)$ has Property $P_{1, 3}$ and the following Property that we call $P^*_{2, 1}$:\newline

Let $u, v \in V(\rho(T_4))$ such that $u \neq v$, $u \neq \overline{v}$ and $\{u, v\} \neq \{\overline{1}, 4\}, \{1, \overline{4}\}, \{\overline{2}, 3\}, \{2, \overline{3}\}$. There exist at least one vertex in $\rho(T_4)$ that is a positive neighbor of both $u$ and $v$.
\end{lemma}

The last lemma can be checked by going through every pair of adjacent vertices in $\rho(T_4)$.

%% file: proof_hex.tex
\section{Proof of Theorem~\ref{thm:hex}}
\label{sec:hex}

In this section, we prove that the chromatic number of signed hexagonal grids equals 4. To get this result, we first prove that 2-edge-colored hexagonal grids admit a homomorphism to the 2-edge-colored graph $\rho(T_4)$. Lemma~\ref{lem:BG} will allow us to conclude.

\begin{lemma}
\label{lem:rhoT4}
Every 2-edge-colored hexagonal grid admits a homomorphism to the 2-edge-colored graph $\rho(T_4)$.
\end{lemma}
\begin{proof}
Let $G$ be a 2-edge-colored hexagonal grid and $s$ be its signature. We want to show that $G$ admits a homomorphism to $\rho(T_4)$. We give two coordinates $i$ and $j$ to each of the vertices of the hexagonal grid according to Figure~\ref{fig:hex_grid}. These coordinates allow us to create an order on the vertex set of $G$ by saying that $v_{i, j} < v_{k, l}$ if $i < k$ or $i = k, j < l$. \newline

\begin{figure}[H]
	\centering
	\scalebox{1}
	{
		\begin{tikzpicture}[thick]
		\def \radius {3cm}
		\def \margin {8} 
		\tikzstyle{vertex}=[circle,minimum width=0.5em]
		
		\node[draw, vertex] (a2) at (1, 0) {};
		\node[draw, vertex] (a3) at (2, 0) {};
		\node[draw, vertex] (a5) at (4, 0) {};
		\node[draw, vertex] (a6) at (5, 0) {};
		
		\node[draw, vertex] (b1) at (0.5, 0.86) {};
		\node[draw, vertex] (b3) at (2.5, 0.86) {};
		\node[draw, vertex] (b4) at (3.5, 0.86) {};
		\node[draw, vertex] (b6) at (5.5, 0.86) {};
		
		\node[draw, vertex] (c2) at (1, 0.86*2) {};
		\node[draw, vertex] (c3) at (2, 0.86*2) {};
        \node[draw, vertex] (c5) at (4, 0.86*2) {};
		\node[draw, vertex] (c6) at (5, 0.86*2) {};
		
		\node[draw, vertex] (d1) at (0.5, 0.86*3) {};
		\node[draw, vertex] (d3) at (2.5, 0.86*3) {};
		\node[draw, vertex] (d4) at (3.5, 0.86*3) {};
		\node[draw, vertex] (d6) at (5.5, 0.86*3) {};
		
		\node[draw, vertex] (e2) at (1, 0.86*4) {};
		\node[draw, vertex] (e3) at (2, 0.86*4) {};
		\node[draw, vertex] (e5) at (4, 0.86*4) {};
		\node[draw, vertex] (e6) at (5, 0.86*4) {};
		
		\draw (e2.north) node[above]{$v_{1, 1}$};
		\draw (e3.north) node[above]{$v_{1, 2}$};
		\draw (e5.north) node[above]{$v_{1, 3}$};
		\draw (e6.north) node[above]{$v_{1, 4}$};
		\draw (c2.north) node[above, shift={(0.06, 0)}]{$v_{3, 1}$};
		\draw (c3.north) node[above, shift={(-0.1, 0)}]{$v_{3, 2}$};
		\draw (c5.north) node[above, shift={(0.06, 0)}]{$v_{3, 3}$};
		\draw (c6.north) node[above, shift={(-0.1, 0)}]{$v_{3, 4}$};
		\draw (a2.north) node[above, shift={(0.06, 0)}]{$v_{5, 1}$};
		\draw (a3.north) node[above, shift={(-0.1, 0)}]{$v_{5, 2}$};
		\draw (a5.north) node[above, shift={(0.06, 0)}]{$v_{5, 3}$};
		\draw (a6.north) node[above, shift={(-0.1, 0)}]{$v_{5, 4}$};
		\draw (d1.north) node[above, shift={(-0.1, 0)}]{$v_{2, 1}$};
		\draw (d3.north) node[above, shift={(0.06, 0)}]{$v_{2, 2}$};
		\draw (d4.north) node[above, shift={(-0.1, 0)}]{$v_{2, 3}$};
		\draw (d6.north) node[above, shift={(0.06, 0)}]{$v_{2, 4}$};
		\draw (b1.north) node[above, shift={(-0.1, 0)}]{$v_{4, 1}$};
		\draw (b3.north) node[above, shift={(0.06, 0)}]{$v_{4, 2}$};
		\draw (b4.north) node[above, shift={(-0.1, 0)}]{$v_{4, 3}$};
		\draw (b6.north) node[above, shift={(0.06, 0)}]{$v_{4, 4}$};
		
		\draw[line width=0.9mm] (a2) -- (a3);
		\draw[line width=0.9mm] (a5) -- (a6);
		\draw[line width=0.9mm] (c2) -- (c3);
		\draw[line width=0.9mm] (c5) -- (c6);
		\draw[line width=0.9mm] (e2) -- (e3);
		\draw[line width=0.9mm] (e5) -- (e6);
		
		\draw[line width=0.9mm] (b3) -- (b4);
		\draw[line width=0.9mm] (d3) -- (d4);
		
		\draw (a2) -- (b1);
		\draw[line width=0.9mm] (a3) -- (b3);
		\draw (a5) -- (b4);
		\draw[line width=0.9mm] (a6) -- (b6);
		
		\draw (c2) -- (d1);
		\draw[line width=0.9mm] (c3) -- (d3);
		\draw (c5) -- (d4);
		\draw[line width=0.9mm] (c6) -- (d6);
		
        \draw[line width=0.9mm] (b1) -- (c2);
        \draw (b3) -- (c3);
        \draw[line width=0.9mm] (b4) -- (c5);
        \draw (b6) -- (c6);
        
        \draw[line width=0.9mm] (d1) -- (e2);
        \draw (d3) -- (e3);
        \draw[line width=0.9mm] (d4) -- (e5);
        \draw (d6) -- (e6);
        
        \draw[dotted] (a2) -- (0.5, -0.86);
        \draw[dotted] (a3) -- (2.5, -0.86);
        \draw[dotted] (a5) -- (3.5, -0.86);
        \draw[dotted] (a6) -- (5.5, -0.86);
        \draw[dotted] (b6) -- (6.5, 0.86);
        \draw[dotted] (d6) -- (6.5, 0.86*3);
		\end{tikzpicture}
	}
	\caption{A hexagonal grid.}
	\label{fig:hex_grid}
\end{figure}
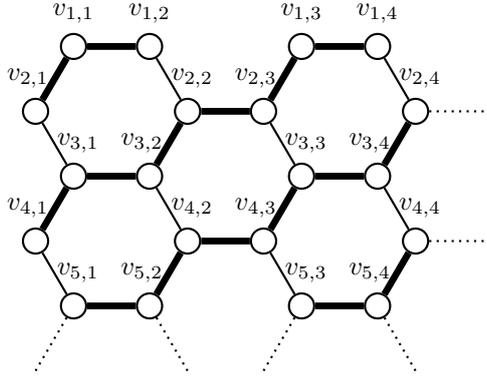 

Let $A$ be a 2-edge-colored graph that admits a homomorphism $\varphi$ to $B$ and let $A'$ be the graph obtained after switching $A$ at a vertex $v$. By Lemma~\ref{lem:BG}, $A'$ admits a homomorphism $\varphi'$ to $I$. It suffices to take $\varphi'$ as follows:
$$
\varphi'(u) = \left\{
    \begin{array}{ll}
        \varphi(u) & \mbox{if } u \neq v \\*[0.5cm]
        \overline{\varphi(v)} & \mbox{otherwise} 
    \end{array}
\right.
$$

Therefore we can, without loss of generality, start by switching $G$ at several vertices such that every edge $v_{i, j} v_{k, l}$ with $i=k$ and $j+1=l$ or $i+1=k$ and $j=l$ are positive (these edges are thicker in Figure~\ref{fig:hex_grid}).

To do that, for every vertex $v_{i, j}$ such that $i+j=1\mod2$ and $i\geq 2$ in the order defined earlier we do the following:
\begin{itemize}
\item If $s(v_{i, j}v_{i-1, j}) = s(v_{i-1, j}v_{i-1, j+1}) = -1$, we switch $G$ at $v_{i-1, j}$.
\item If $s(v_{i, j}v_{i-1, j}) = 1$ and  $s(v_{i-1, j}v_{i-1, j+1}) = -1$, we switch $G$ at $v_{i, j}$ and $v_{i-1, j}$.
\item If $s(v_{i, j}v_{i-1, j}) = -1$ and $s(v_{i-1, j}v_{i-1, j+1}) = 1$, we switch $G$ at $v_{i, j}$.
\end{itemize} 

We now create a homomorphism $\varphi$ from $G$ to $\rho(T_4)$ by coloring each vertex in the order defined earlier. We partition the vertices of $G$ into two sets $V_1$ and $V_2$. In $V_1$ we put every vertex $v_{i, j}$ such that $i+j \equiv 0 \mod 2$. In $V_2$ we put all the other vertices.

When coloring a vertex $v_{i, j}$ in $V_1$, such a vertex is adjacent to one already colored vertex (unless $i = 1$ in which case it is trivial to color $v_{i,j}$). Therefore, Property $P_{1, 3}$ of $\rho(T_4)$ tells us that there are at least $3$ possible colors for $v_{i, j}$ with respect to its already colored neighbor. Note that among these 3 available colors, there is no antitwins. Thanks to this remark, it is always possible to choose one color as follows: 

\begin{itemize}
\item If $\varphi(v_{i-1, j+1}) = 1$ or $\overline{1}$, $\varphi(v_{i, j}) \notin \{1, \overline{1}, 4, \overline{4}\}$;
\item If $\varphi(v_{i-1, j+1}) = 4$ or $\overline{4}$, $\varphi(v_{i, j}) \notin \{1, \overline{1}, 4, \overline{4}\}$;
\item If $\varphi(v_{i-1, j+1}) = 2$ or $\overline{2}$, $\varphi(v_{i, j}) \notin \{2, \overline{2}, 3, \overline{3}\}$;
\item If $\varphi(v_{i-1, j+1}) = 3$ or $\overline{3}$, $\varphi(v_{i, j}) \notin \{2, \overline{2}, 3, \overline{3}\}$;
\end{itemize}

When coloring a vertex $v_{i, j}$ in $V_2$, such a vertex is adjacent to two already colored vertex (unless $i = 1$ or $j=1$ in which cases it is trivial to color $v_{i,j}$). Vertex $v_{i-1, j}$ belongs to $V_1$ and we can therefore use $P^*_{2, 1}$, thanks to the restrictions on $\varphi(v_{i-1, j})$ defined earlier, to find a color for each vertex in $V_2$.
\end{proof}

We use Lemmas~\ref{lem:BG} and~\ref{lem:rhoT4} to prove that the chromatic number of signed hexagonal grids is at most 4. Note that a cycle on 6 vertices with exactly one negative edges needs at least 4 color to be colored \cite{degmax}. Therefore, the chromatic number of signed hexagonal grids is 4.

%% file: proof_tri.tex
\section{Proof of Theorem~\ref{thm:tri}}
\label{sec:tri}
In this section, we prove that the chromatic number of signed triangular grids is at most 10. To get this result, we first prove that 2-edge-colored triangular grids admit a homomorphism to the 2-edge-colored graph $\rho(SP_9^+)$. Lemma~\ref{lem:BG} will allow us to conclude.

\begin{lemma}
\label{lemline width=0.7mmSP9+}
Every 2-edge-colored triangular grid admits a homomorphism to the 2-edge-colored graph $\rho(SP_9^+)$.
\end{lemma}

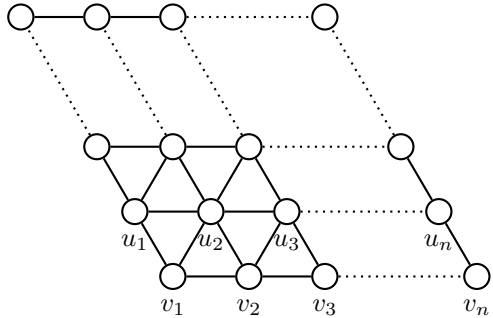
\begin{figure}[H]
	\centering
	\scalebox{1}
	{
		\begin{tikzpicture}[thick]
		\def \radius {3cm}
		\def \margin {8} 
		\tikzstyle{vertex}=[circle,minimum width=0.5em]
		
		\node[draw, vertex] (a1) at (1, 0) {};
		\node[draw, vertex] (a2) at (2, 0) {};
		\node[draw, vertex] (a3) at (3, 0) {};

		\node[draw, vertex] (a5) at (5, 0) {};
		
		\node[draw, vertex] (b1) at (0.5, 0.86) {};
		\node[draw, vertex] (b2) at (1.5, 0.86) {};
		\node[draw, vertex] (b3) at (2.5, 0.86) {};

		\node[draw, vertex] (b5) at (4.5, 0.86) {};
		
		\node[draw, vertex] (c1) at (0, 0.86*2) {};
		\node[draw, vertex] (c2) at (1, 0.86*2) {};
		\node[draw, vertex] (c3) at (2, 0.86*2) {};

		\node[draw, vertex] (c5) at (4, 0.86*2) {};
		
		\node[draw, vertex] (e1) at (-1, 0.86*4) {};
		\node[draw, vertex] (e2) at (0, 0.86*4) {};
		\node[draw, vertex] (e3) at (1, 0.86*4) {};

		\node[draw, vertex] (e5) at (3, 0.86*4) {};
		
		\draw (a1.south) node[below]{$v_1$};
		\draw (a2.south) node[below]{$v_2$};
		\draw (a3.south) node[below]{$v_3$};
		\draw (a5.south) node[below]{$v_n$};
		
		\draw (b1.south) node[below]{$u_1$};
		\draw (b2.south) node[below]{$u_2$};
		\draw (b3.south) node[below]{$u_3$};
		\draw (b5.south) node[below]{$u_n$};
		
		\draw (a1) -- (a2);
		\draw (a2) -- (a3);
	    
	    \draw (b1) -- (b2);
		\draw (b2) -- (b3);
	    
	    \draw (c1) -- (c2);
		\draw (c2) -- (c3);
	    
	    \draw (e1) -- (e2);
		\draw (e2) -- (e3);
	    
	    \draw (a1) -- (b1);
		\draw (a2) -- (b2);
		\draw (a3) -- (b3);
		\draw (a5) -- (b5);
	    
	    \draw (b1) -- (c1);
		\draw (b2) -- (c2);
		\draw (b3) -- (c3);
	    \draw (b5) -- (c5);
	    
		\draw (a1) -- (b2);
		\draw (a2) -- (b3);
	    
	    \draw (b1) -- (c2);
		\draw (b2) -- (c3);
		
		\draw[dotted] (a3) -- (a5);
		\draw[dotted] (b3) -- (b5);
		\draw[dotted] (c3) -- (c5);
		\draw[dotted] (e3) -- (e5);
		
		\draw[dotted] (c1) -- (e1);
		\draw[dotted] (c2) -- (e2);
		\draw[dotted] (c3) -- (e3);
		\draw[dotted] (c5) -- (e5);
		\end{tikzpicture}
	}
	\caption{A triangular grid.}
	\label{fig:tri_grid}
\end{figure}
\begin{proof}

Remember that $SP_9^+$ is $SP_9$ with an added vertex $\infty$ that is positively adjacent to every other vertex. Let $G$ be a 2-edge-colored triangular grid and $s$ be the signature of $G$. We proceed by induction on the horizontal rows of $G$ as depicted in Figure~\ref{fig:tri_grid}. Note that the first row of $G$ is trivial to color.
Let $G'$ be $G$ without the last row. By the induction hypothesis, there is a homomorphism $\varphi'$ from $G'$ to $\rho(SP_9^+)$. We now show that we can extend this homomorphism to a homomorphism $\varphi$ from the whole graph $G$ to $\rho(SP_9^+)$.

Let $v_1, v_2, ..., v_n$ be the vertices of the last row of $G$ and $u_1, u_2, ..., u_n$ be the vertices of the second to last row of $G$ (the last row of $G'$). See Figure~\ref{fig:tri_grid}.

By Property $P_{2, 4}$ of $\rho(SP_9^+)$, we can find two colors (or even four but we only need two) to color $v_1$ with respect to its already colored neighbors ($u_1$ and $u_2$). Note that we do not take the color of $u_3$ into account (yet).

Without loss of generality, we can assume that $u_2 u_3$ is a positive edge, $\varphi'(u_2) = 0$ and $\varphi'(u_3) = 1$ because $\rho(SP_9^+)$ is edge-transitive and antiautomorphic.\newline

Suppose $s(u_2 v_1) = s(u_2 v_2) = s(u_3 v_2) = s(v_1 v_2) = +1$.
The four colors available for $v_2$ by Property $P_{2, 4}$ of $\rho(SP_9^+)$ with respect to the colors of $u_2$ and $u_3$ are $2, \infty, \overline{x+2}$ and $\overline{2x+2}$.

Since $\varphi'(u_2) = 0$ and $s(u_2 v_1) = +1$, the two colors available for $v_1$ belong to the set: $$\{1, 2, x, \overline{x+1}, \overline{x+2}, 2x, \overline{2x+1}, \overline{2x+2}, \infty\}$$

If $v_1$ is colored in $1$, $v_2$ can be colored in $2, \infty, \overline{x+2}$ or $\overline{2x+2}$.

If $v_1$ is colored in $2$, $v_2$ can be colored in $\infty$.

If $v_1$ is colored in $x$, $v_2$ can be colored in $\infty$ or $\overline{x+2}$.

If $v_1$ is colored in $\overline{x+1}$, $v_2$ can be colored in $2$ or $\overline{x+2}$.

If $v_1$ is colored in $\overline{x+2}$, $v_2$ can be colored in $\overline{2x+2}$.

If $v_1$ is colored in $2x$, $v_2$ can be colored in $\infty$ or $\overline{x+2}$.

If $v_1$ is colored in $\overline{2x+1}$, $v_2$ can be colored in $2$ or $\overline{2x+2}$.

If $v_1$ is colored in $\overline{2x+2}$, $v_2$ can be colored in $\overline{x+2}$.

If $v_1$ is colored in $\infty$, $v_2$ can be colored in $2$.\newline

We can now see that any pair of vertices in $\{1, 2, x, \overline{x+1},$ $ \overline{x+2}, 2x, \overline{2x+1}, \overline{2x+2}, \infty\}$ allows $v_2$ to be colored in at least two colors.\newline

We can proceed in a similar manner with the following three cases and arrive to the same conclusion:
\begin{itemize}
\item $s(u_2 v_1) = -1$, $s(u_2 v_2) = s(u_3 v_2) = s(v_1 v_2) = +1$,

\item $s(u_2 v_1) = +1$, $s(u_2 v_2) = -1$,  $s(u_3 v_2) = s(v_1 v_2) = +1$,

\item $s(u_2 v_1) = s(u_2 v_2) = +1$,  $s(u_3 v_2) = -1$,  $s(v_1 v_2) = +1$.
\end{itemize}

By Lemma~\ref{lem:BG}, each of these four cases also accounts for $3$ other cases: the signature obtained after switching at $v_1$, $v_2$ and both $v_1$ and $v_2$. We have therefore covered all 16 ($2^4$) possible signatures of $u_2v_1, u_2 v_2, u_3 v_2$ and $v_1 v_2$.

Therefore, $v_2$ can be colored in at least 2 colors. Similarly, we can find at least two colors for $v_3$ and so on until $v_n$. Finally, we can arbitrarily choose one of these two colors for $v_n$, accordingly choose a color for $v_{n-1}$ and so on to get a homomorphism $\varphi$ from $G$ to $\rho(SP_9^+)$.
\end{proof}

We conclude the proof of Theorem~\ref{thm:tri} by using Lemmas~\ref{lemline width=0.7mmSP9+} and~\ref{lem:BG}. \newline

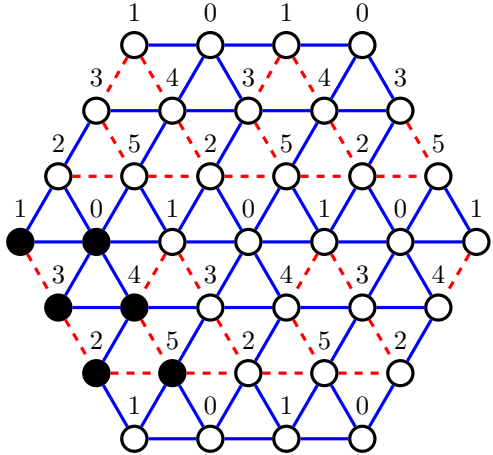
\begin{figure}
	\centering
	\scalebox{1}
	{
		\begin{tikzpicture}[very thick]
		\def \radius {3cm}
		\def \margin {8} 
		\tikzstyle{vertex}=[circle,minimum width=0.5em]
		
		\node[draw, vertex] (0) at (-1.5, 3*0.87) {};
		\node[draw, vertex] (1) at (-0.5, 3*0.87) {};
		\node[draw, vertex] (2) at (0.5, 3*0.87) {};
		\node[draw, vertex] (3) at (1.5, 3*0.87) {};
		
		\node[draw, vertex] (4) at (-2, 2*0.87) {};
		\node[draw, vertex] (5) at (-1, 2*0.87) {};
		\node[draw, vertex] (6) at (0, 2*0.87) {};
		\node[draw, vertex] (7) at (1, 2*0.87) {};
		\node[draw, vertex] (8) at (2, 2*0.87) {};
		
		\node[draw, vertex] (9) at (-2.5, 0.87) {};
		\node[draw, vertex] (10) at (-1.5, 0.87) {};
		\node[draw, vertex] (11) at (-0.5, 0.87) {};
		\node[draw, vertex] (12) at (0.5, 0.87) {};
		\node[draw, vertex] (13) at (1.5, 0.87) {};
		\node[draw, vertex] (14) at (2.5, 0.87) {};
		
		\node[draw, vertex, fill] (15) at (-3, 0) {};
		\node[draw, vertex, fill] (16) at (-2, 0) {};
		\node[draw, vertex] (17) at (-1, 0) {};
		\node[draw, vertex] (18) at (0, 0) {};
		\node[draw, vertex] (19) at (1, 0) {};
		\node[draw, vertex] (20) at (2, 0) {};
		\node[draw, vertex] (21) at (3, 0) {};
		
		\node[draw, vertex, fill] (22) at (-2.5, -0.87) {};
		\node[draw, vertex, fill] (23) at (-1.5, -0.87) {};
		\node[draw, vertex] (24) at (-0.5, -0.87) {};
		\node[draw, vertex] (25) at (0.5, -0.87) {};
		\node[draw, vertex] (26) at (1.5, -0.87) {};
		\node[draw, vertex] (27) at (2.5, -0.87) {};
		
		\node[draw, vertex, fill] (28) at (-2, -2*0.87) {};
		\node[draw, vertex, fill] (29) at (-1, -2*0.87) {};
		\node[draw, vertex] (30) at (0, -2*0.87) {};
		\node[draw, vertex] (31) at (1, -2*0.87) {};
		\node[draw, vertex] (32) at (2, -2*0.87) {};
		
		\node[draw, vertex] (33) at (-1.5, -3*0.87) {};
		\node[draw, vertex] (34) at (-0.5, -3*0.87) {};
		\node[draw, vertex] (35) at (0.5, -3*0.87) {};
		\node[draw, vertex] (36) at (1.5, -3*0.87) {};
		
		\draw (0.north) node[above]{$1$};
		\draw (1.north) node[above]{$0$};
		\draw (2.north) node[above]{$1$};
		\draw (3.north) node[above]{$0$};
		\draw (4.north) node[above]{$3$};
		\draw (5.north) node[above]{$4$};
		\draw (6.north) node[above]{$3$};
		\draw (7.north) node[above]{$4$};
		\draw (8.north) node[above]{$3$};
		\draw (9.north) node[above]{$2$};
		\draw (10.north) node[above]{$5$};
		\draw (11.north) node[above]{$2$};
		\draw (12.north) node[above]{$5$};
		\draw (13.north) node[above]{$2$};
		\draw (14.north) node[above]{$5$};
		\draw (15.north) node[above]{$1$};
		\draw (16.north) node[above]{$0$};
		\draw (17.north) node[above]{$1$};
		\draw (18.north) node[above]{$0$};
		\draw (19.north) node[above]{$1$};
		\draw (20.north) node[above]{$0$};
		\draw (21.north) node[above]{$1$};
		\draw (22.north) node[above]{$3$};
		\draw (23.north) node[above]{$4$};
		\draw (24.north) node[above]{$3$};
		\draw (25.north) node[above]{$4$};
		\draw (26.north) node[above]{$3$};
		\draw (27.north) node[above]{$4$};
		\draw (28.north) node[above]{$2$};
		\draw (29.north) node[above]{$5$};
		\draw (30.north) node[above]{$2$};
		\draw (31.north) node[above]{$5$};
		\draw (32.north) node[above]{$2$};
		\draw (33.north) node[above]{$1$};
		\draw (34.north) node[above]{$0$};
		\draw (35.north) node[above]{$1$};
		\draw (36.north) node[above]{$0$};
		
		\draw[color=blue] (0) -- (1);
		\draw[color=blue] (1) -- (2);
		\draw[color=blue] (2) -- (3);
		\draw[color=blue] (4) -- (5);
		\draw[color=blue] (5) -- (6);
		\draw[color=blue] (6) -- (7);
		\draw[color=blue] (7) -- (8);
		\draw[dashed, color=red] (9) -- (10);
		\draw[dashed, color=red] (10) -- (11);
		\draw[dashed, color=red] (11) -- (12);
		\draw[dashed, color=red] (12) -- (13);
		\draw[dashed, color=red] (13) -- (14);
		\draw[color=blue] (15) -- (16);
		\draw[color=blue] (16) -- (17);
		\draw[color=blue] (17) -- (18);
		\draw[color=blue] (18) -- (19);
		\draw[color=blue] (19) -- (20);
		\draw[color=blue] (20) -- (21);
		\draw[color=blue] (22) -- (23);
		\draw[color=blue] (23) -- (24);
		\draw[color=blue] (24) -- (25);
		\draw[color=blue] (25) -- (26);
		\draw[color=blue] (26) -- (27);
		\draw[dashed, color=red] (28) -- (29);
		\draw[dashed, color=red] (29) -- (30);
		\draw[dashed, color=red] (30) -- (31);
		\draw[dashed, color=red] (31) -- (32);
		\draw[color=blue] (33) -- (34);
		\draw[color=blue] (34) -- (35);
		\draw[color=blue] (35) -- (36);
		
		\draw[dashed, color=red] (15) -- (22);
		\draw[dashed, color=red] (22) -- (28);
		\draw[color=blue] (28) -- (33);
		\draw[color=blue] (9) -- (16);
		\draw[color=blue] (16) -- (23);
		\draw[dashed, color=red] (23) -- (29);
		\draw[color=blue] (29) -- (34);
		\draw[dashed, color=red] (4) -- (10);
		\draw[color=blue] (10) -- (17);
		\draw[dashed, color=red] (17) -- (24);
		\draw[dashed, color=red] (24) -- (30);
		\draw[color=blue] (30) -- (35);
		\draw[dashed, color=red] (0) -- (5);
		\draw[dashed, color=red] (5) -- (11);
		\draw[color=blue] (11) -- (18);
		\draw[color=blue] (18) -- (25);
		\draw[dashed, color=red] (25) -- (31);
		\draw[color=blue] (31) -- (36);
		\draw[color=blue] (1) -- (6);
		\draw[dashed, color=red] (6) -- (12);
		\draw[color=blue] (12) -- (19);
		\draw[dashed, color=red] (19) -- (26);
		\draw[dashed, color=red] (26) -- (32);
		\draw[dashed, color=red] (2) -- (7);
		\draw[dashed, color=red] (7) -- (13);
		\draw[color=blue] (13) -- (20);
		\draw[color=blue] (20) -- (27);
		\draw[color=blue] (3) -- (8);
		\draw[dashed, color=red] (8) -- (14);
		\draw[color=blue] (14) -- (21);
		
		\draw[dashed, color=red] (0) -- (4);
		\draw[color=blue] (4) -- (9);
		\draw[color=blue] (9) -- (15);
		\draw[color=blue] (1) -- (5);
		\draw[color=blue] (5) -- (10);
		\draw[color=blue] (10) -- (16);
		\draw[color=blue] (16) -- (22);
		\draw[dashed, color=red] (2) -- (6);
		\draw[color=blue] (6) -- (11);
		\draw[color=blue] (11) -- (17);
		\draw[dashed, color=red] (17) -- (23);
		\draw[color=blue] (23) -- (28);
		\draw[color=blue] (3) -- (7);
		\draw[color=blue] (7) -- (12);
		\draw[color=blue] (12) -- (18);
		\draw[color=blue] (18) -- (24);
		\draw[color=blue] (24) -- (29);
		\draw[color=blue] (29) -- (33);
		\draw[color=blue] (8) -- (13);
		\draw[color=blue] (13) -- (19);
		\draw[dashed, color=red] (19) -- (25);
		\draw[color=blue] (25) -- (30);
		\draw[color=blue] (30) -- (34);
		\draw[color=blue] (14) -- (20);
		\draw[color=blue] (20) -- (26);
		\draw[color=blue] (26) -- (31);
		\draw[color=blue] (31) -- (35);
		\draw[dashed, color=red] (21) -- (27);
		\draw[color=blue] (27) -- (32);
		\draw[color=blue] (32) -- (36);

		\end{tikzpicture}
	}
	\caption{A signed triangular grid in which every $C_4$ is unbalanced colored with 6 colors.}
	\label{fig:conject}
\end{figure}

We say that a cycle is unbalanced if it has an odd number of negative edges. Every $C_4$ in the signed triangular grid from Figure~\ref{fig:conject} is unbalanced and it can be colored with $6$ colors (note that the resulting target graph is $SP_5^+$). We can easily extend this construction to create a signed triangular grid of any size such that every $C_4$ is unbalanced and it can be colored with $6$ colors. To do so, we can repeat a motif made of six vertices in all directions (the black vertices in Figure~\ref{fig:conject} represent this motif). By Proposition 3.2 of \cite{Z}, this means that every signed triangular grid such that every $C_4$ is unbalanced can be colored with $6$ colors. This is of particular interest because when coloring an unbalanced $C_4$, every vertex in the cycle must have different identities (this is not the case with a $C_4$ that is not unbalanced) and therefore a graph in which every $C_4$ is unbalanced is especially hard to color with few colors. Therefore, we conjecture the following:

\begin{conjecture}
The chromatic number of signed triangular grids is $6$.
\end{conjecture}

Note that the chromatic number of signed triangular grids is at least $6$ since a wheel on $7$ vertices such that every $C_4$ is unbalanced cannot be colored with $5$ colors. To prove it, let $G$ be a $2$-edge-colored wheel with vertex set $\{u, u_1, ..., u_6\}$ and center $u$, and let $T$ be an antitwinned graph of vertex set $\{0, \overline{0}, 1, \overline{1}, ..., 4, \overline{4}\}$ where $\overline{i}$ is the antitwin of $i$ such that $G$ admits a homomorphism $\varphi$ to $T$. By Proposition 3.2 of \cite{Z}, it is possible to switch some of the vertices of $G$ such that the outer cycle alternates between positive and negative edges and every edge incident to $u$ is positive. Assume w.l.o.g. that $\varphi(u) = 0$ (note that we can relabel the vertices of $T$ if needed). We now show that we cannot color a pair of vertices $(v_1, v_2)$ of $G$ with antitwins. Suppose $v_1$ and $v_2$ have colors that are antitwinned. If $v_1$ or $v_2=u$ then we have a contradiction because $v_1$ and $v_2$ are adjacent. Otherwise, $v_1$ and $v_2$ are both positively adjacent to $u$ which also gives us a contradiction. Therefore we can assume w.l.o.g. that we do not need to use colors $\overline{0}, \overline{1}, ..., \overline{4}$. W.l.o.g. let $\varphi(u_1)=1$, $\varphi(u_2)=2$ and $\varphi(u_3)=3$. We can collor $u_4$ in either $1$ or $4$. Suppose we color $u_4$ with $1$. We have to color $u_5$ in $4$ and we then cannot color $u_6$. Suppose we color $u_4$ with $4$. We can color $u_5$ with $1$ or $2$ but both possibilities do not allow us to color $u_6$.  Therefore $G$ does not admit a homomorphism to $T$. We conclude by using Lemma~\ref{lem:BG}.

%% file: main.bbl
\begin{thebibliography}{10}

\bibitem{HomSG}
R.~Naserasr, E.~Rollov\'{a}, and \'{E}. Sopena.
\newblock Homomorphisms of signed graphs.
\newblock {\em Journal of Graph Theory}, 79(3):178--212, 2015.

\bibitem{HomSG2}
Reza Naserasr, Éric Sopena, and Thomas Zaslavsky.
\newblock Homomorphisms of signed graphs: An update.
\newblock {\em European Journal of Combinatorics}, 91:103222, 2021.
\newblock Colorings and structural graph theory in context (a tribute to Xuding
  Zhu).

\bibitem{OPS}
Pascal Ochem, Alexandre Pinlou, and Sagnik Sen.
\newblock Homomorphisms of 2-edge-colored triangle-free planar graphs.
\newblock {\em Journal of Graph Theory}, 85(1):258--277, 2017.

\bibitem{Hom2ec}
Amanda Montejano, Pascal Ochem, Alexandre Pinlou, Andr{\'e} Raspaud, and Eric
  Sopena.
\newblock Homomorphisms of 2-edge-colored graphs.
\newblock {\em Discrete Applied Mathematics}, 158(12):1365--1379, 2010.

\bibitem{BPS}
Julien Bensmail, Sandip Das, Soumen Nandi, Th{\'e}o Pierron, Sagnik Sen, and
  Eric Sopena.
\newblock On the signed chromatic number of some classes of graphs.
\newblock {\em arXiv preprint arXiv:2009.12059}, 2020.

\bibitem{Grid}
Janusz Dybizbański, Anna Nenca, and Andrzej Szepietowski.
\newblock Signed coloring of 2-dimensional grids.
\newblock {\em Information Processing Letters}, 156:105918, 2020.

\bibitem{conj}
Julien Bensmail, Soumen Nandi, Mithun Roy, and Sagnik Sen.
\newblock Classification of edge-critical underlying absolute planar cliques
  for signed graphs.
\newblock {\em The Australasian Journal of Combinatorics}, 77(1):117--135,
  2020.

\bibitem{HomEG}
R.~C. Brewster and T.~Graves.
\newblock Edge-switching homomorphisms of edge-coloured graphs.
\newblock {\em Discrete Mathematics}, 309(18):5540 -- 5546, 2009.
\newblock Combinatorics 2006, A Meeting in Celebration of Pavol Hell’s 60th
  Birthday (May 1–5, 2006).

\bibitem{OPS16}
P.~Ochem, A.~Pinlou, and S.~Sen.
\newblock Homomorphisms of 2-edge-colored triangle-free planar graphs.
\newblock {\em Journal of Graph Theory}, 85(1):258--277, 2017.

\bibitem{degmax}
C.~Duffy, F.~Jacques, M.~Montassier, and A.~Pinlou.
\newblock On the signed chromatic number of some classes of graphs.
\newblock ar{X}iv:2009.12059, submitted, 2020.

\bibitem{Z}
T.~Zaslavsky.
\newblock Signed graphs.
\newblock {\em Discrete Applied Mathematics}, 4(1):47 -- 74, 1982.

\end{thebibliography}
